\documentclass[a4paper,10pt, leqno]{amsart}
\date{July 22, 2020}
\usepackage{amsmath, amssymb, graphics, setspace}

\usepackage[dvips]{graphicx}

\title[Curved foldings with common creases 
and crease patterns 
]{%
Curved foldings with common 
creases  and crease patterns}

\author{A.~Honda}
\address[Atsufumi Honda]{
Department of Applied Mathematics, 
Faculty of Engineering, Yokohama National University,
79-5 Tokiwadai, Hodogaya, Yokohama 240-8501, Japan
}
\email{honda-atsufumi-kp@ynu.ac.jp}

\author{K.~Naokawa}
\address[Kosuke Naokawa]{%
Department of Computer Science, 
Faculty of Applied Information Science,
Hiroshima Institute of Technology,  
2-1-1 Miyake, Saeki, Hiroshima, 731-5193, Japan
}
\email{k.naokawa.ec@cc.it-hiroshima.ac.jp}

\author{K. Saji}
\address[Kentaro Saji]{%
  Department of Mathematics,
  Faculty of Science,
  Kobe University,
  Rokko, Kobe 657-8501}
\email{saji@math.kobe-u.ac.jp}
\author{M. Umehara}
\address[Masaaki Umehara]{%
  Department of Mathematical and Computing Sciences,
  Tokyo Institute of Technology,
  Tokyo 152-8552, Japan}
\email{umehara@is.titech.ac.jp}

\author{K. Yamada}
\address[Kotaro Yamada]{%
  Department of Mathematics,
  Tokyo Institute of Technology,
  Tokyo 152-8551, Japan}
\email{kotaro@math.titech.ac.jp}

\thanks{%
The first author was partially supported by 
Grant-in-Aid for Early-Career Scientists
 No.~19K14526. The second author was partially supported by 
Grant-in-Aid for Young Scientists (B) No.~17K14197,
and the third author
was 
partially supported by  Grant-in-Aid for 
Scientific Research (C) No.\ 18K03301
The fifth author 
was partially 
supported by Grant-in-Aid for 
Scientific Research (B) No.\ 17H02839.
}%
\newcommand{\op}[1]{{\operatorname{#1}}}

\newcommand{\R}{\boldsymbol{R}}

\newcommand{\mb}[1]{{\mathbf #1}}

\renewcommand{\phi}{\varphi}
\renewcommand{\epsilon}{\varepsilon}

\renewcommand{\det}{\op{det}}
\ifx\EMS\undefined

\else
\fi

\numberwithin{equation}{section}
\newtheorem{Theorem}{Theorem}[section]
\newtheorem{Proposition}[Theorem]{Proposition}
\newtheorem{Corollary}[Theorem]{Corollary}
\newtheorem{Lemma}[Theorem]{Lemma}

\theoremstyle{definition}

\newtheorem{Remark}[Theorem]{Remark}

\newtheorem*{acknowledgments}{Acknowledgments}

 \makeatletter
       \def\@makefnmark{%
               \leavevmode
               \raise.9ex\hbox{\check@mathfonts
                       \fontsize\sf@size\z@\normalfont%
                               \@thefnmark}%
       }
       \makeatother

\begin{document}
\maketitle
\begin{abstract}
Consider a curve $\Gamma$ in a domain $D$ in the plane $\R^2$.
Thinking of $D$ as a piece of paper, one can make a curved folding $P$ 
in the Euclidean space $\R^3$.
The singular set $C$ of $P$ as a space curve
is called the {\it crease} of $P$ and the initially given plane curve $\Gamma$
is called the {\it crease pattern} of $P$.
In this paper, we show that in general there are four 
distinct non-congruent curved foldings 
with a given pair consisting of a crease 
and crease pattern.
Two of these possibilities were already known,
but it seems that the other two possibilities (i.e.
four possibilities in total) are presented here for the first time.
\end{abstract}

\begin{figure}[htb]
\begin{center}
\includegraphics[height=3.2cm]{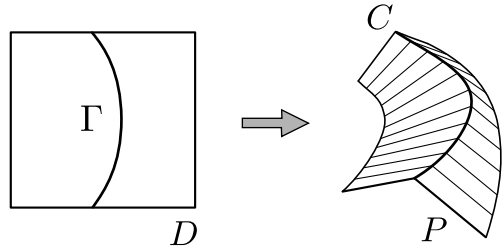}\qquad
\caption{A given crease pattern $\Gamma(\subset D)$
in $\R^2$  (left)
 and its realization $P$ as a curved folding along the  crease $C$
(right).}\label{fig:Ori}
\end{center}
\end{figure}

\section*{Introduction}
Let $C$ be a space curve with positive 
curvature embedded in $\R^3$, and let $\Gamma$ be a curve
with non-vanishing curvature function (i.e. the curve $\Gamma$
has no inflections) embedded in a plane $\R^2$.
We assume that the maximum  
of the absolute value of the curvature function of $\Gamma$ is
less than the minimum of the curvature function of $C$.
We are interested in curved foldings
whose crease is $C$ and crease pattern is $\Gamma$
(cf. Figure \ref{fig:Ori}).  
There is numerous literature in this subject, for example,
\cite{DDHKT1,DDHKT2,DD, FT0, Hu, K} are fundamental references.
 In this paper, we focus on curved foldings 
which are produced from a single curve $\Gamma$ 
in $\R^2$.  It is then natural to consider the following question:

\medskip
\noindent
{\bf Question.}
{\it How many curved foldings can we have with the
same crease and crease pattern?
}

\medskip
For a given curved folding, 
it is well-known that there is an {\it adjacent curved folding}
with the same creases and crease patterns
(see Figure \ref{fig:2} and also \cite{FT0}, \cite[Pages 417, 419]{PW}),
which is based on the fact that there are
two solutions for developable surfaces
whose geodesic curvatures along $C$ 
are equal to the curvature of $\Gamma$ (cf.~\eqref{eq:beta}
and \eqref{eq:beta3}). 
Our purpose here is to give an answer to Question.
We prove the following:

\medskip
\noindent
{\bf  Theorem A.}
{\it The number of such curved foldings  
is at most four, that is, there are at most four possibilities for
curved foldings with a given pair of
crease and crease pattern}.

\medskip
Roughly speaking, the idea is as follows:
The above two solutions of developable surfaces
depend on the orientation of
crease pattern as a plane curve. 
We then reverse the orientation of the
curve $C$ and obtain totally four possibilities.
Moreover, we also show the following:

\medskip
\noindent
{\bf  Theorem B.}
{\it For a suitable choice of crease $C$ and 
crease pattern $\Gamma$,
the four curved foldings associated
with $C$ and $\Gamma$ are 
non-congruent, that is,
no two of them can be made equal by
any isometry of the Euclidean $3$-space $\R^3$. }

\section{Preliminaries}

\subsection{fundamental properties of developable strips}
We fix an interval $I=[a,b]$ ($a<b$) in $\R$ and
let $\mb c(t)$ ($t\in I$) be a $C^\infty$-differentiable 
curve embedded in $\R^3$ parametrized by 
arc-length, and denote by $\kappa(t):=|\mb c''(t)|$  
its curvature function, where $\mb c'':=d^2\mb c/dt^2$.
Throughout this paper, we assume $\kappa>0$.
The vector fields 
$$
\mb e:=\mb c',\quad \mb n:=(1/\kappa)\mb c'', \quad
\mb b:=\mb e\times \mb n
$$ 
along $C:=\mb c(I)$ are called
the {\it unit tangent vector field},
 {\it unit principal normal vector field} and
{\it unit bi-normal vector field} of $\mb c$, respectively,
where $\times$ denotes the vector product of $\R^3$.
In this paper, the torsion function $\tau$ of $\mb c$
is defined by 
$
\tau:=\mb n'\cdot \mb b,
$	
where the dot denotes the inner product of  $\R^3$.
We remark that the sign of  $\tau$
is opposite to that in \cite{FT0}.

We fix a positive number $\epsilon>0$.
To explain curved foldings, 
we consider a ruled surface
$
f:I \times (-\epsilon, \epsilon) \to \R^3 
$ 
written in the form
\begin{equation}\label{eq:F}
f(t,v):=\mb c(t)+v \xi(t) \qquad (t\in I,
\,\, |v|<\epsilon),
\end{equation}
where $\xi(t)$ is a vector field along $C$
given by 
\begin{equation}\label{eq:xi0}
\xi(t):=
\cos \beta(t) \mb e(t)+
\sin \beta(t)\Big(
\cos \alpha(t) \mb n(t)+\sin \alpha(t) \mb b(t)\Big),
\end{equation}
and $\alpha(t)$ and $\beta(t)$
are $C^\infty$-functions satisfying
\begin{equation}\label{eq:angle}
0<|\alpha(t)|<\frac{\pi}{2},\qquad 0<\beta(t)<\pi
\qquad (t\in I).
\end{equation}
We call such a $\xi$ the {\it normalized ruling vector field} of $f$.
If $\epsilon$ is sufficiently small, $f$ gives an
embedding.  

The functions $\alpha$ and $\beta$ are 
called the {\it first}  and {\it second angular functions}
of $f$, respectively.
A ruled surface $f$ is called a {\it developable strip} if
its Gaussian curvature $K$ vanishes identically.
It is well-known that $K=0$ is equivalent to
the condition $\det(\mb c',\xi,\xi')=0$ (cf. \cite[Appendix B-4]{UY-book}).
So $f$ is developable if and only if
\begin{equation}\label{eq:beta}
\cot \beta(t)=\frac{\alpha'(t)+\tau(t)}{\kappa(t)\sin \alpha(t)}
\end{equation}
holds.
In particular, $f$ is uniquely determined by the first angular function $\alpha$.
So we express $f$ as $f^\alpha$ when we wish to emphasize that $f$ is
associated with $\alpha$. 
The unit normal vector $\nu$ and unit co-normal vector ${\mb N}$ of 
$f^\alpha$ along $C$ are given by
\begin{align*}\label{eq:n}
\nu(t):=-\sin \alpha(t) \mb n(t)+\cos \alpha(t) \mb b(t),\qquad 
\mb N(t):=\cos \alpha(t) \mb n(t)+\sin \alpha(t) \mb b(t),
\end{align*}
respectively. 
So the geodesic curvature $\mu_f$ of $\mb c$ as a curve on $f$ is 
given by
\begin{equation}\label{eq:mu}
\mu_f=\mb c''\cdot \mb N
=\kappa\mb n\cdot (\cos \alpha \mb n+\sin \alpha \mb b)
=\kappa\cos \alpha.
\end{equation}
Using the Frenet-Serret formula (cf. \cite[(5.4)]{UY-book}),
we have
\begin{equation}\label{eq:nu-D}
\nu'=\kappa \sin \alpha \,\mb e-(\alpha'+\tau)\mb N.
\end{equation}
We let $ds^2=Edt^2+2Fdtdv+Gdv^2$ be the first fundamental form
of $f$ along the curve $C$. Then we have
\begin{equation}\label{eq:EFG}
E=
\Big(\sin\beta-v (\beta'+\kappa_g)\Big)^2+\cos^2\beta,\quad
F=\cos \beta,\quad G=1.
\end{equation}
In particular, it only depends on $\beta$ and $\kappa_g$.
Substituting $v=0$, we set
$$
E(t):=1,\quad F(t):=\cos \beta(t), \quad G(t):=1.
$$
On the other hand, using \eqref{eq:nu-D}, the second 
fundamental form $L(t)dt^2+2M(t)dtdv+N(t)dv^2$ 
of $f$ along the curve $C$ can be written  as follows: 
 \begin{align*}
&L(t)=-f_t(t,0)\cdot \nu'(t)=-\mb e\cdot \nu'=-\kappa \sin \alpha, \\
&M(t)=-f_t(t,0)\cdot \nu_v(t)=0,\quad
N(t)=-f_{v}(t,0)\cdot \nu_v(t)=0.
\end{align*}
So, the mean curvature function $H_f(t)$ of $f$ 
along  $C$ can be computed as (cf. \cite[(9.2)]{UY-book})
\begin{equation}\label{eq:H}
H_f
=-\frac{\kappa\sin\alpha\,\op{cosec}^2\beta}{2}=
-\frac{\kappa^2\sin^2 \alpha
+(\alpha'+\tau)^2}{2\kappa\sin \alpha}.
\end{equation}
Since $\sin \alpha(t)\ne 0$ ($t\in I$), 
the mean curvature function does not have
any zeros, and so $f$ has no umbilics along $C$.
In particular, the ruling direction (i.e. the asymptotic direction)
of $f$ is uniquely determined.

\subsection{Properties of Curved foldings}

We let $f:=f^\alpha$ be a developable strip with
the first angular function $\alpha$.
Since $f$ can be developed to a plane $\R^2$, 
the curve $C$ corresponds to a curve $\Gamma_f$ in $\R^2$.
Since the developing procedure preserves the
geodesic curvature $\mu_f$ of $C$, the curvature
function of the plane curve $\Gamma_f$ is
given by $\mu_f(=\kappa\cos \alpha)$ (cf. \eqref{eq:mu}).
We call $\Gamma_f$ the {\it generator} of $f$.

Conversely, we give a procedure to construct developable strips from
a given plane curve as follows:
Let $\gamma:I\to \R^2$ be an embedded curve
parametrized by arc-length without inflection points.
We set $\Gamma:=\gamma(I)$.
We assume the curvature function $\mu(t)$ 
of $\gamma(t)$ satisfies
\begin{equation}\label{eq:star}
0<\mu(t)<\kappa(t) \qquad (t\in I).
\end{equation}
We then define the function $\alpha:I\to (0,\pi/2)$
(cf. \eqref{eq:mu}) so that
\begin{equation}\label{eq:beta2}
\mu(t)=\kappa(t)\cos \alpha(t).
\end{equation}
Then the geodesic curvature  $\mu_f(t)$ of
$f:=f^\alpha$ coincides with the curvature function $\mu(t)$
of $\Gamma(t)$.
Here we examine other possibility for the
existence of a developable surface whose geodesic curvature is $\mu(t)$.
By \eqref{eq:beta2}, only one other possibility can be considered, which is
the case that the first angular function is $-\alpha(t)$. 
So we set 
$$
\check f:=f^{-\alpha},
$$
and call it the {\it dual} of $f$. 
Since the first angular function of the dual $\check f$
is $-\alpha$, the geodesic curvature $\mu_{\check f}$
of $\check f$ along $C$ is $\mu(t)$ as same as $f$, and
the second angular function $\check \beta$
of $\check f$ is given by (cf. \eqref{eq:beta})
\begin{equation}\label{eq:beta3}
\cot \check \beta(t)=\frac{\alpha'(t)-\tau(t)}{\kappa(t)\sin \alpha(t)}.
\end{equation}
If $C$ lies in a plane, then $\tau$ vanishes identically, and
\begin{equation}\label{eq:bb}
\beta=\check \beta
\end{equation}
holds. 
In particular, the initial plane curve $\Gamma$ 
can be considered as a common generator of $f$ and $\check f$.
We then set
$$
\phi_f(t,v):=
\begin{cases}
f(t,v) & (v\ge 0), \\
\check f(t,v) & (v<0), 
\end{cases}\qquad
\psi_f(t,v):=
\begin{cases}
\check f(t,v) & (v\ge 0), \\
f(t,v) & (v< 0), 
\end{cases}
$$
and call them {\it origami-maps} associated with
the developable surface $f$.
By definition
\begin{equation}\label{eq:p-p}
\phi_{f}=\psi_{\check f},\qquad \phi_{\check f}=\psi_{f}
\end{equation}
hold. 
The curve
$\Gamma$ is called the {\it crease pattern\/} of $\phi_f$
and $\psi_f$, which is also the generator of $f$ and $\check f$.
We call the image $P$ of the origami-map $\phi_f$
a {\it curved folding} along $C$ associated with $f$. 
Moreover, the image $\check P$ of 
the origami-map $\phi_{\check f}(=\psi_f)$
is called the {\it adjacent curved folding} 
associated with $P$. 
As in \cite{FT0}, $P\cap \check P=C$ holds,
and $P\cup \check P$ coincides with the union of
the images of $f$ and $\check f$
(these properties are mentioned in the introduction).
In Figure \ref{fig:2} (left, center),
a pair of adjacent curved foldings
along a circle is given. 
Actual curved foldings can be regarded as
the images of origami-maps.

\begin{Remark}
Let $\check P$ be the adjacent curved folding of $P$.
Then $P$ and $\check P$ 
can be isometrically deformed to subsets of $\R^2$ 
preserving its crease pattern, 
and the second angular functions $\beta$ and $\check \beta$
(cf. \cite[Theorem 3.1]{DDHKT2}).
Such a deformation is
called a {\it rigid-ruling folding motion} 
and are discussed in \cite{DDHKT2}
for a collection of curves. 
\end{Remark}

We prepare a lemma:

\begin{Lemma}\label{lem:LC}
Consider the set of intersections 
between $\phi_f$ and $\psi_f$ defined by
\begin{align*}
S:=\Big\{(t,v)\in \Omega_\epsilon \,;\, 
& \text{there exists $(t',v')\in \Omega_\epsilon\setminus\{(t,v)\}$} \\
& \phantom{aaaaaaaaaaaa}\text{such that $\phi_f(u,v)=\psi_f(t',v')$}
\Big\},
\end{align*}
where $\Omega_\epsilon:=I\times (-\epsilon,\epsilon)$.
Then $S$ has no interior points
for sufficiently small $\epsilon(>0)$.
\end{Lemma}

\begin{proof}
By \eqref{eq:H} and \eqref{eq:angle}, the mean curvature function
of $f$ never vanishes.
In particular, $f$ has no umbilics and $\xi(t):=f_v(t,0)$ along $C$
coincides with the (uniquely determined) asymptotic direction of $f$.
By the same reason, the ruling direction $\check \xi(t)$ of $\check f$ 
is also determined uniquely. We set
$$
\Omega_\epsilon^+:=I\times (0,\epsilon),
\quad
\Omega_\epsilon^-:=I\times (-\epsilon,0).
$$
Suppose that $S$ contains a non-empty open subset $W$ of $\Omega_\epsilon$.
Since $f$ is an embedding
for sufficiently small $\epsilon(>0)$,
$f(\Omega_\epsilon^+)$ cannot meet $f(\Omega_\epsilon^-)$.
In particular, $f(W)$ is a subset of $f(\Omega_\epsilon)$.
On the other hand,  since the first angular function of $\check f$
is $-\alpha$, the normalized ruling vector field  of $\check f$
is linearly independent of that of  $f$, which
contradicts the fact $f(W)\subset f(\Omega_\epsilon)$.
So the conclusion is obtained.
\end{proof}

An isometry $T$ of $\R^3$ 
which is not the identity map
is called a {\it symmetry} of $C$ if
$T(C)=C$.

\begin{Corollary}\label{cor:LC}
Suppose that $C$ has no symmetries. Then for any  
$\delta\in (0,\epsilon]$, the image $\phi_f(\Omega_\delta)$ 
cannot be congruent to a subset of $\psi_f(\Omega_\epsilon)$.
\end{Corollary}

\begin{proof}
If not, there exist $\delta>0$ and an isometry $T$ of $\R^3$
such that $T\circ \phi_f(\Omega_\delta)$
is a subset of  $\psi_f(\Omega_\epsilon)$.
Since $T$ maps the crease of $\phi_f$ to
that of $\psi_f$, it fixes the curve $C$.
Since $C$ has no symmetries, $T$ must be the identity map.
Then we have
$\phi_f(\Omega_\delta)\subset
\psi_f(\Omega_\epsilon)$, which 
contradicts Lemma \ref{lem:LC}.
\end{proof}

We remark that there is an actual procedure to construct a pair of 
adjacent curved foldings as follows:  
Preparing two sheets of paper, we first draw a crease pattern on one sheet.
We next draw the same crease pattern also on the other sheet.
We then lay one sheet upon the other, and
staple these sheets together along the common crease pattern.
After that, we isometrically deform this
two joined sheets as in Figure \ref{fig:2} (right), where one can see 
a pair of curved foldings.

\begin{figure}[htb]
\begin{center}
\includegraphics[height=1.9cm]{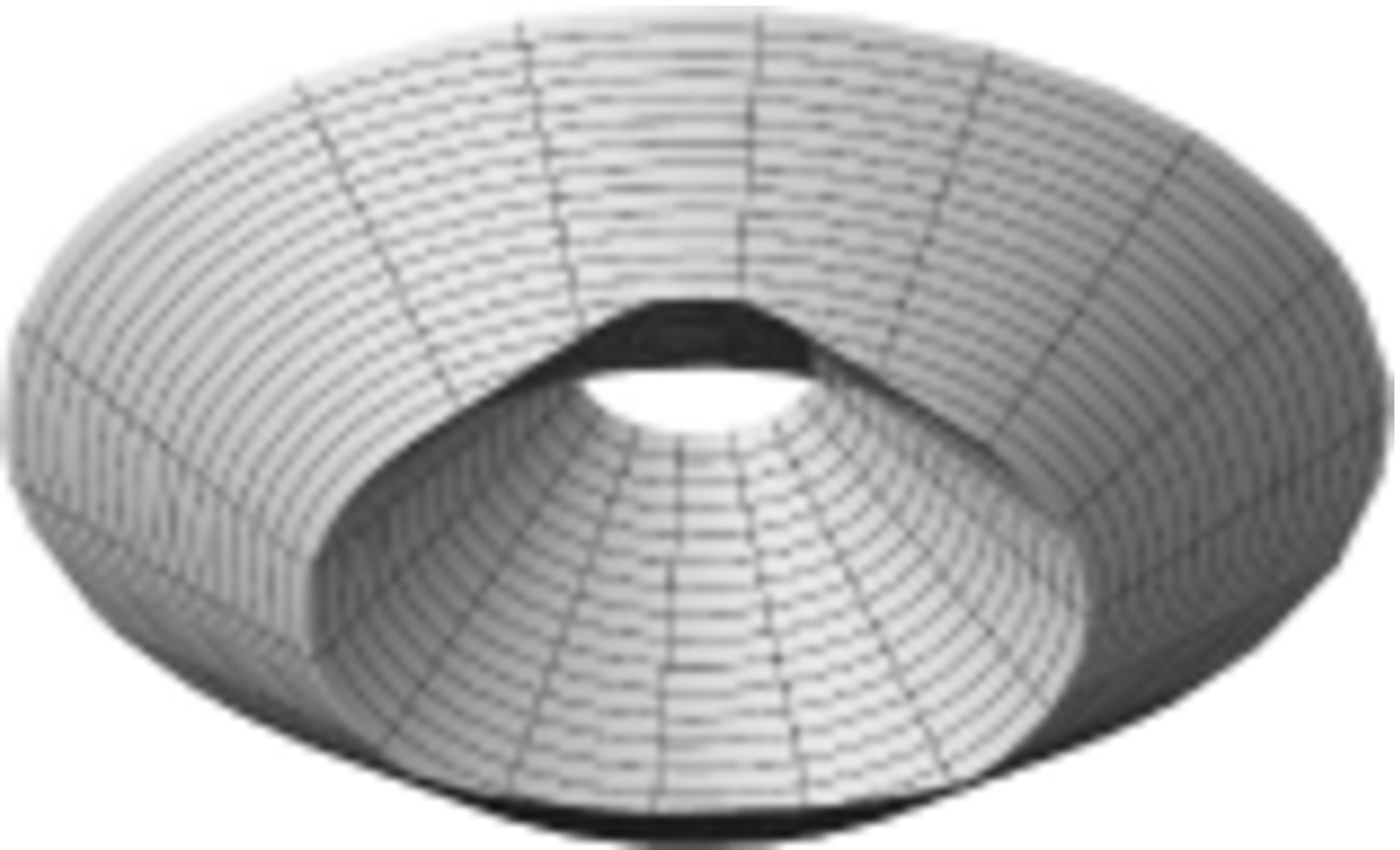}\,\,
\includegraphics[height=1.9cm]{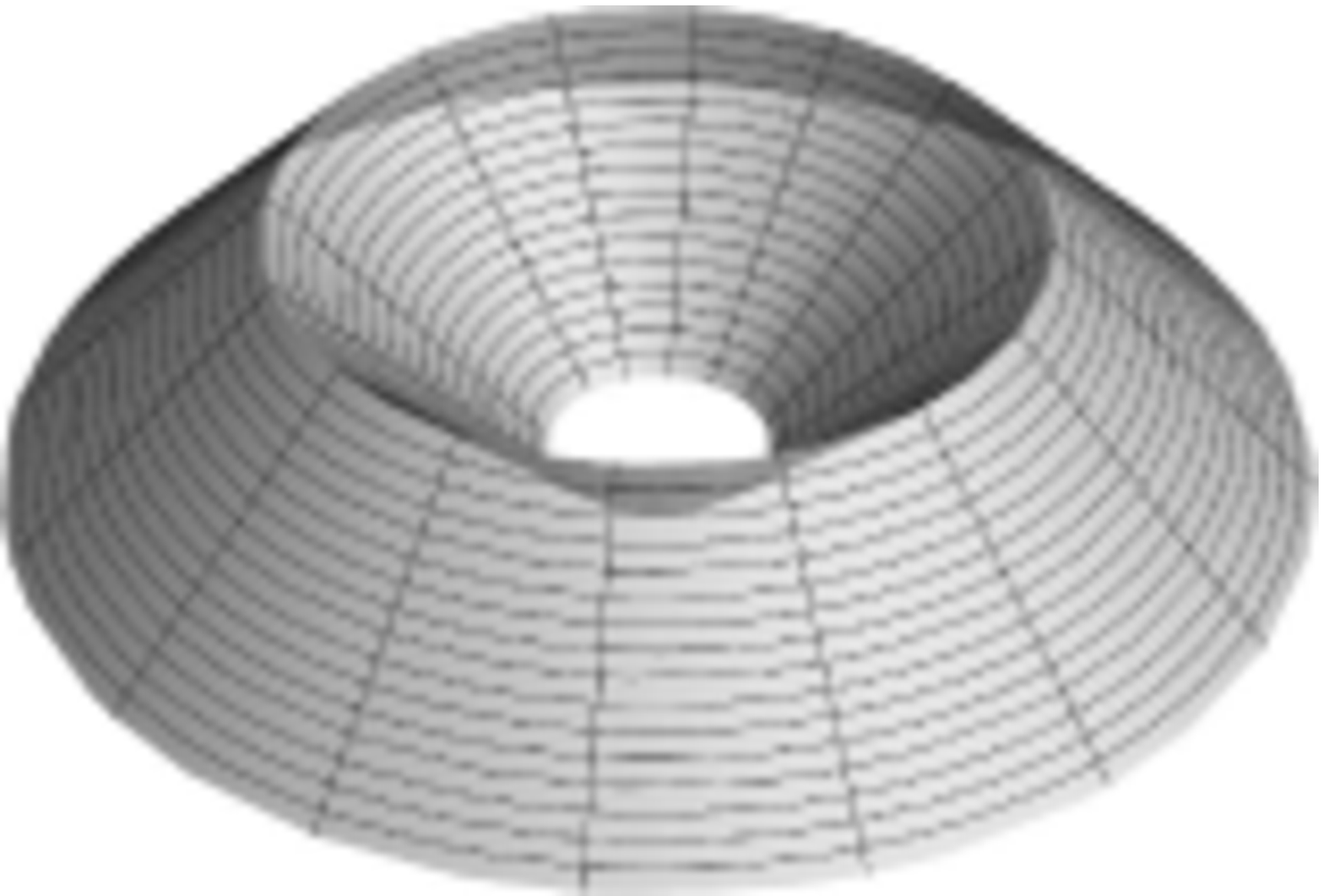}\qquad
\includegraphics[height=2.2cm]{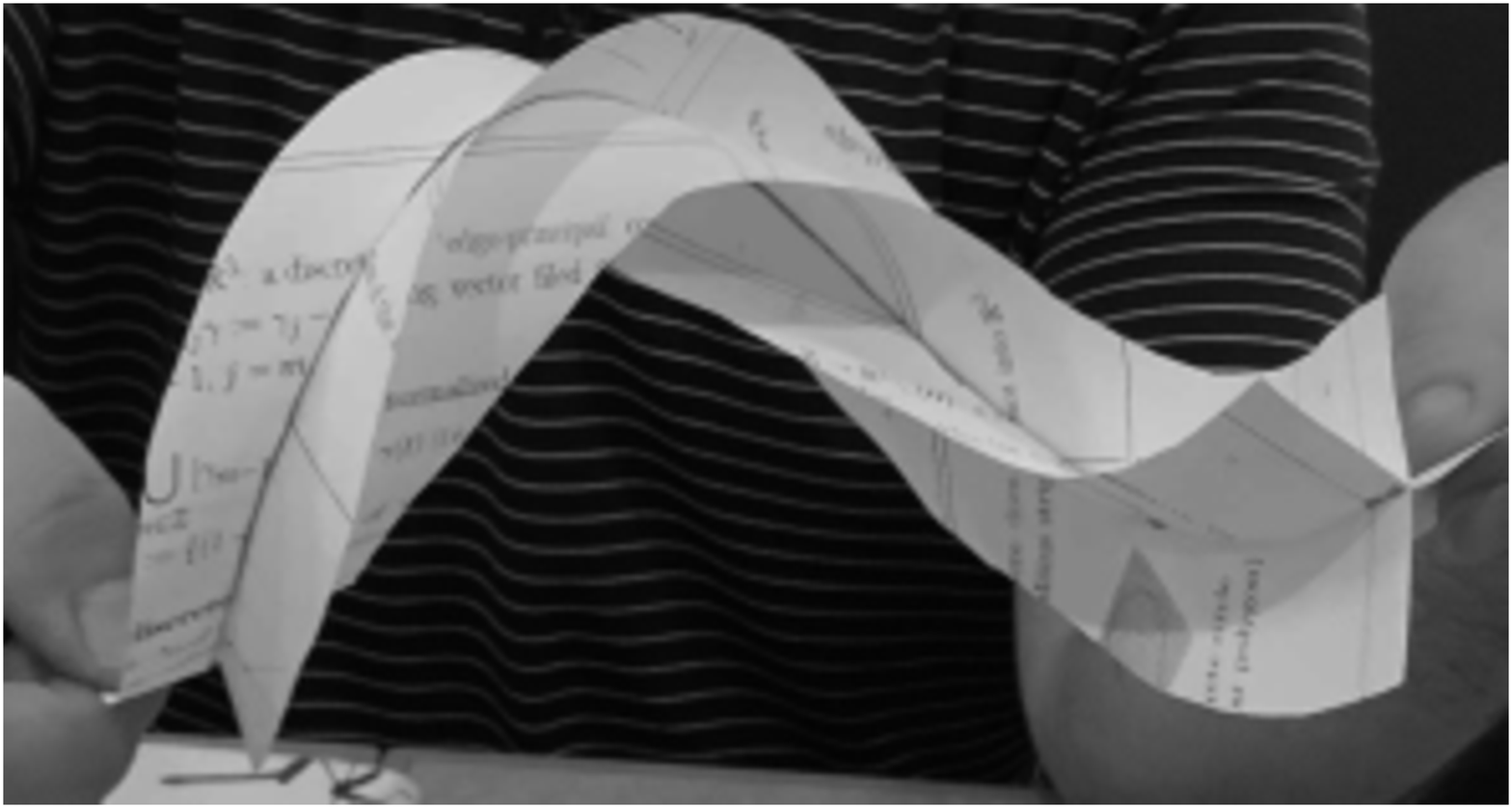}
\caption{A pair of two curved foldings along a circle (left, center)
and along an actual pair of adjacent curved foldings (right).}\label{fig:2}
\end{center}
\end{figure}

\section{Proof of Theorem A}
We set $I:=[a,b]$ $(a<b)$ and let $\gamma(t)$ ($t\in I$)
be a regular curve in $\R^2$ with arc-length parameter, which
has no inflection points. We set $\Gamma:=\gamma(I)$.
Let $\mu(t)$ be the curvature function of $\gamma(t)$.
Without loss of generality, we may assume $\mu>0$.
As in the introduction, we assume that the maximum  
of the absolute value of the curvature function of $\Gamma$ is
less than the minimum of the curvature function of $C$.
So we have
\begin{equation}\label{eq:cg}
\max_{t\in I} \left| \mu(t)\right| <\min_{t\in I}\kappa(t),
\end{equation} where $\kappa(t)$ is the curvature function of ${\mb c}(t)$.
We define
$
\alpha:I\to (0,\pi/2)
$
by
$$
\cos \alpha(t)=\frac{\mu(t)}{\kappa(t)},
$$
and we denote by $f(t,v)$ (resp. $\check f(t,v)$) the developable
strip along $\mb c(t)$ whose first angular function
is $\alpha(t)$ (resp. $-\alpha(t)$). 
We know that $f$ and $\check f$ are
only the possibility of developable strips along $\mb c(t)$
whose geodesic curvatures along $\mb c$ coincide with $\mu(t)$.
However, if we reverse the orientation of $C$, 
we can find other two possibilities of 
developable strips whose geodesic curvatures coincide with $\mu(t)$:
The curvature function $\kappa^\sharp(t)$ 
and the torsion function $\mu^\sharp(t)$ of 
$\mb c^\sharp (t):=\mb c(-t+a+b)$ is given by
$$
\kappa^\sharp(t)=\kappa(-t+a+b),\qquad \tau^\sharp(t)=\tau(-t+a+b),
$$
respectively.
There exists a unique function
$
\alpha_*:I\to (0,\pi/2)
$
such that
$$
\cos \alpha_*(t):=\frac{\mu(t)}{\kappa^\sharp(t)}.
$$
We let $f_*(t,v)$ (resp. $\check f_*(t,v)$)
be the developable strip along $\mb c^\sharp(t)$
whose first angular function is 
$\alpha_*(t)$ (resp. $-\alpha_*(t)$).
Then, by definition, the geodesic curvature of 
$f_*$ (resp. $\check f_*$)
along $\mb c^\sharp$
with respect to $f_*$ (resp.  $\check f_*$) is $\mu(t)$.
We call $f_*$ the {\it inverse} of $f$, and call the dual $\check f_*$ of $f_*$
the {\it inverse dual} of $f$. It can be easily checked that
$\check f_*$ coincides with the inverse of $\check f$.
The second angular function $\beta_*(t)$ (resp. $\check \beta_*(t)$)
of $f_*$ (resp.  $\check f_*$) satisfies
\begin{align*}
\cot \beta_*(t):=\frac{\alpha'_*(t)+\tau^\sharp(t)}{\kappa^\sharp(t)\sin \alpha_*(t)},
\quad
\left( \text{resp.}\,\,
\cot \check \beta_*(t):=\frac{\alpha'_*(t)-\tau^\sharp(t)}
{\kappa^\sharp(t)\sin \alpha_*(t)}\right).
\end{align*}
We denote by $P,\check P$, $P_*$ and $\check P_*$,
the images of $\phi_f$, $\phi_{\check f}(=\psi_f)$,
$\phi_{f_*}$ and  $\phi_{\check f_*}(=\psi_{f_*})$, respectively.
By this construction, 
 $P,\check P, P_*$ and $\check P_*$ have the same crease $C$ and
crease pattern $\Gamma$. 
Thus, to prove Theorem A, it is sufficient to show the
following assertion:

\begin{Proposition}\label{Prop:at4}
The number of curved foldings 
with the same crease and crease pattern
is at most four.
\end{Proposition}

\begin{proof}
The two boundary points of $\Gamma$ correspond to 
those of $C$. So only the ambiguity to make
developable surfaces along $C$ fixing $\Gamma$
comes from the orientation of $C$. 
In each orientation of $C$, there is 
an ambiguity of the choice of signs of the first angular function.
So we obtain the conclusion.
\end{proof}

\section{Proof of Theorem B}

We set $I_0:=[1/10,9/10]$ and consider a space curve
$$
\mb c_0(t):=\left(\arctan t,
\frac{\log \left(1+t^2\right)}{\sqrt{2}},
t-\arctan t\right)
\qquad (t\in I_0),
$$
where $t$ gives an arc-length parameter.
We set $C_0:=\mb c_0(I_0)$. 
The curvature function $\kappa$ and
torsion function $\tau$ of $\mb c_0$ are
computed as
$$
\kappa(t)=\tau(t)=\frac{\sqrt{2}}{1+t^2}.
$$
Since $\kappa(t)\ne \kappa(1-t)$,
$C_0$ has no symmetries in $\R^3$
because of the fundamental theorem for space curves
(cf. \cite[Theorem 5.2]{UY-book}).
We then consider a function
$
\alpha(t):={\pi(t+10)}/24.
$
We let $\check f,f_*$ and $\check f_*$ be 
the three developable strips along $C_0$ 
induced by $f\,(=f^\alpha)$.
The generators of $f,\check f,f_*$ and $\check f_*$ 
coincides with the plane curve $\Gamma_0$ whose curvature
function is $\kappa(t)\cos \alpha(t)$.
For each $\delta\in (0,\epsilon]$, we set
$$
P_1(\delta)=\phi_{f}(\Omega_\delta),\quad
P_2(\delta)=\phi_{\check f}(\Omega_\delta),\quad
P_3(\delta)=\phi_{f_*}(\Omega_\delta),\quad
P_4(\delta)=\phi_{\check f_*}(\Omega_\delta).
$$
To prove Theorem B, it is sufficient to
show the following:

\begin{figure}[htb]
\begin{center}
\includegraphics[width=0.4\linewidth]{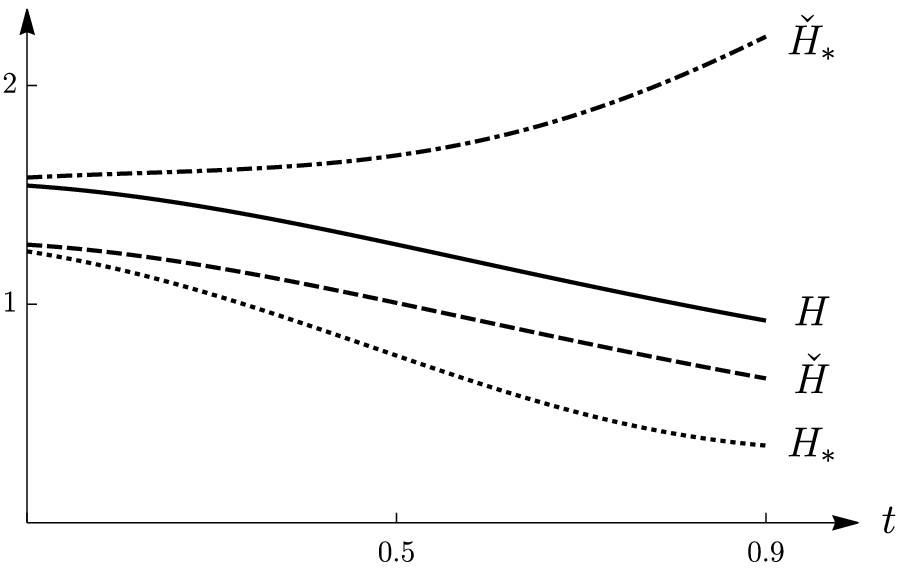}\qquad
\raisebox{2ex}{\includegraphics[width=0.5\linewidth]{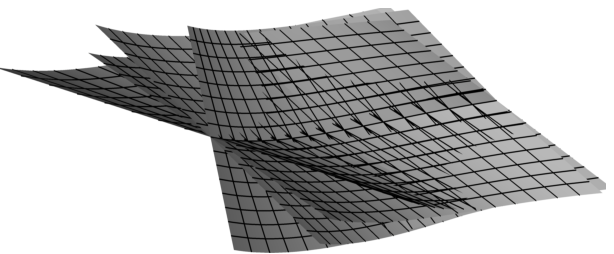}}
\caption{The graph of $H,\,\, \check H,\,\, H_*$
 and $\check H_{*}$
and the figures of four curved foldings
$\{P_i(\epsilon)\}_{i=1,\dots,4}$ for $\epsilon=1/5$. 
}\label{fig:4}
\end{center}
\end{figure}

\begin{Proposition}
Fix $\delta\in (0,\epsilon]$. 
Let $i,j$ be numbers belonging to $\{1,2,3,4\}$.
If an isometry $T$ of $\R^3$ satisfies
$T(P_i(\delta))\subset P_j(\epsilon)$, then $i$ must be
equal to $j$. 
\end{Proposition}

\begin{proof}
Suppose that the assertion fails.
By Lemma \ref{lem:LC}, $P_1(\delta)$ (resp. $P_2(\delta)$)
cannot be a subset of $P_2(\epsilon)$ (resp. $P_1(\epsilon)$),
and also $P_3(\delta)$ (resp. $P_4(\delta)$)
cannot be a subset of $P_4(\epsilon)$ (resp. $P_3(\epsilon)$).
So we may assume that $i=1,2$
and $j=3,4$ without loss of
generality.
Since $T(P_i(\delta))\subset P_j(\epsilon)$,
we have $T(C_0)=C_0$.
Since $C_0$ has no symmetry, $T$ must be the identity map.
So we have $P_i(\delta) \subset P_j(\epsilon)$.
To make it easy to see the subscriptions,
we set (cf. \eqref{eq:H})
\begin{align*}
&H(t):=|H_f(t)|,\quad
\check H(t):=|H_{\check f}(t)|,\\
&H_*(t):=|H_{f_*}(1-t)|,\quad
\check H_*(t):=|H_{\check f_*}(1-t)|.
\end{align*}
Using \eqref{eq:H}, one can compute that  
$H$, $\check H$,
$H_{*}$ and $\check H_{*}$ 
and can plot the graphs 
as indicated in Figure \ref{fig:4} (left).
Then we can observe the relation
$$
H_*<\check H< H< \check H_*
$$
hold on the interval $[1/10,9/10]$.
Thus, the images of the four surfaces are mutually distinct.
So we obtain the conclusion.
\end{proof}

In Figure \ref{fig:4} (right), the images of
the four surfaces $f,\,\check f,\,f_*$ and $\check f_*$
are indicated, which are all distinct.
In Figure \ref{fig:5}, the crease patterns with ruling directions
of the four curved foldings  
$\{P_i(\epsilon)\}_{i=1,\dots,4}$ for $\epsilon=
1/5$
are drawn. The difference of ruling angles in these four patterns
also implies that the four curved foldings have distinct images 
(the non-congruence of them does not follow from this fact
directly).

\begin{figure}[htb]
\begin{center}
\includegraphics[height=2.0cm]{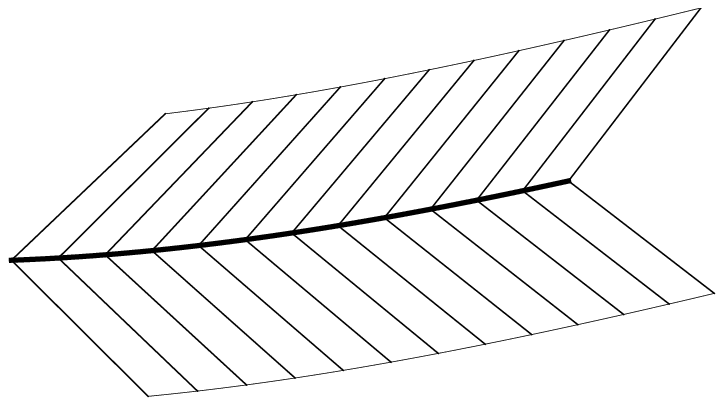}\qquad
\includegraphics[height=2.0cm]{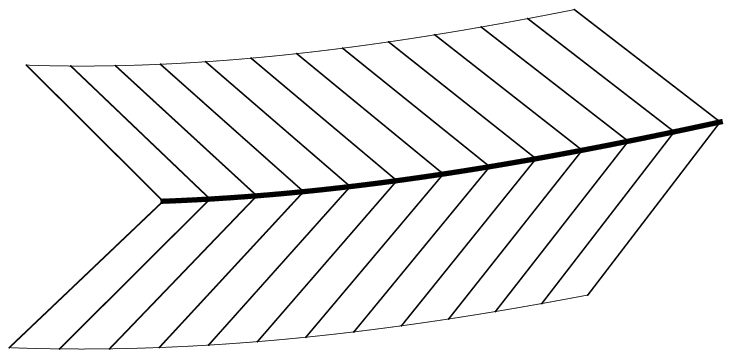} \\
\includegraphics[height=2.0cm]{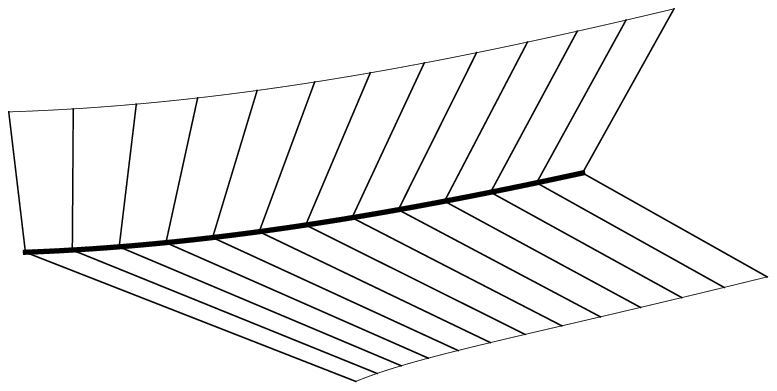}\qquad
\includegraphics[height=2.0cm]{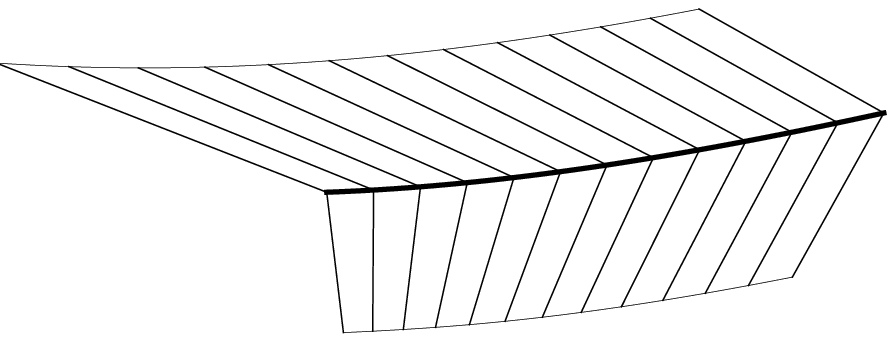} 
\caption{
Crease patterns with ruling directions
for $P_1(\epsilon)$ (upper-left), $P_2(\epsilon)$ (upper-right), $P_3(\epsilon)$ (lower-left)
and $P_4(\epsilon)$ (lower-right) for $\epsilon=1/5$.}\label{fig:5}
\end{center}
\end{figure}

\section{Final remarks}

\begin{itemize}
\item If $C$ or $\Gamma$ has a symmetry, 
the associated four curved foldings may not
be non-congruent. 
For example, we set $I:=[-\pi/6,\pi/6]$ and
consider a part of the unit circle 
$
\mb c_1(t)=(\cos t, \sin t,0)
$
($t\in I$) and set $C_1:=\mb c_1(I)$.
We fix a $C^\infty$-function 
$\alpha(t)$ ($t\in I$) satisfying $0<\alpha<\pi/2$
arbitrarily.
We let $f(t,v)$ be the developable strip along $C_1$
whose first angular function is $\alpha(t)$.
Since $C_1$ is fixed under two isometric transformations
$S:(x,y,z)\mapsto (x,y,-z)$
and $T:(x,y,z)\mapsto (x,-y,z)$,
we have
$\check f:=S\circ f$,  $f_*:=T\circ f$ and 
$\check f_*:=S\circ T\circ f$. 
So they are congruent each other. On the other hand, the four surfaces
$f,\check f,f_*$ and $\check f_*$ have distinct images 
whenever $\alpha(t)$ is not an odd function.
\item 
If the crease $C$ is a closed space curve, then
we have infinitely many
curved foldings along $C$ with
a common crease pattern $\Gamma$.
In fact, if $C$ is non-closed, then
the endpoints of $\Gamma$ correspond to
those of $C$. However, if $C$ is closed,
we have the freedom to assign the base point of 
$\Gamma$ to an arbitrary fixed point on $C$.
Moreover, these infinitely many curved
foldings are mutually non-congruent in 
general, see \cite{HNSUY} for details.
\end{itemize}

\begin{acknowledgments}
The authors thank Professors Jun Mitani, 
Sergei Tabachnikov and Wayne Rossman for valuable comments.
\end{acknowledgments}

\end{document}